\documentclass{amsart}
\usepackage{amsthm}
\usepackage{enumitem}
\usepackage{amssymb}
\usepackage{bbm}
\usepackage{relsize}
\usepackage{eucal}
\usepackage{url}
\usepackage{dsfont}

\usepackage{mathrsfs}
\usepackage[all, arc]{xy}
\usepackage[perpage]{footmisc}
\usepackage{enumerate}
\usepackage{lettrine}
\usepackage[vflt]{floatflt}
\usepackage[lite, numeric,sorted,traditional-quotes,compressed-cites]{amsrefs}

\newcommand{\id}[1]{{\text{id}}_{#1}}
\newcommand{\Alt}[2]{{\rm Alt}^{#1}(#2)}
\newcommand{\Sym}[2]{{\text{Sym}}^{#1}(#2)}
\newcommand{\rep}[2]{{\text{Rep}_{#1}(#2)}}
\newcommand{\fraka}{{\mathfrak A}}
\newcommand{\induced}[3]{{\rm Ind}^{#1}_{#2}(#3)}

\newcommand{\Z}{{\mathbbm Z}}
\newcommand{\gr}[1]{{\text{gr}^{#1}}}

\newcommand{\dime}[1]{{\text{dim}(#1)}}

\newcommand{\cl}[1]{{\text{cl}_{#1}}}
\newcommand{\class}[2]{{\rm cl}_{#1}({#2})}
\renewcommand{\id}[1]{{\rm id}_{#1}}

\newcommand{\trace}[1]{{\rm tr}(#1)}

\newtheoremstyle{statement}
{13pt}
{13pt}
{\it}
{}
{\bf}
{.$-$}
{.5em}
{}

\theoremstyle{statement}

\newtheorem{theorem}{Theorem}[section]
\newtheorem{lemma}[theorem]{Lemma}
\newtheorem{proposition}[theorem]{Proposition}

\newtheorem{cor}[theorem]{Corollary}

\newtheoremstyle{definition}
{13pt}
{13pt}
{}
{}
{\scshape}
{.}
{.5em}
{}

\theoremstyle{definition}

\newtheorem{definition}[theorem]{Definition}

\newtheorem{example}[theorem]{Example}
\newtheorem{remark}[theorem]{Remark}

\newtheoremstyle{definitionprime}
{13pt}
{13pt}
{}
{}
{\scshape}
{$'$.}
{.5em}
{}

\theoremstyle{definitionprime}

\newtheoremstyle{remarks}
{13pt}
{13pt}
{}
{}
{\scshape}
{.}
{.5em}
{}

\theoremstyle{remarks}

\newtheoremstyle{remarksb}
{13pt}
{13pt}
{}
{}
{\scshape}
{.}
{\newline}
{}

\theoremstyle{remarksb}

\newtheoremstyle{underlined}
{13pt}
{13pt}
{\sl}
{}
{\scshape}
{}
{.5em}
{}

\theoremstyle{underlined}

\makeatletter 
\def\@seccntformat#1{\protect\makebox[0pt][r]{\@ifundefined{#1@cntformat}%
   {\csname the#1\endcsname.\quad}%
   {\csname #1@cntformat\endcsname}%
}}
\def\section@cntformat{\S\thesection.\ }
\def\subsection@cntformat{}

\makeatother

\begin{document}
\title{On rings and categories of general representations}


\author{Shahram Biglari}
\address{Fakult\"at f\"ur Mathematik, Universit\"at Bielefeld, D-33615, Bielefeld, Germany}
\curraddr{} \email{biglari@mathematik.uni-bielefeld.de}
\thanks{}




\date{2010}

\begin{abstract} In this note we show that similar to the classical case the ring of representations of symmetric groups in a tensor derived category is certain ring of symmetric functions. We also show that in the general setting considered here, the Adams operations compute the characteristic series associated to powers of endomorphisms.  
\end{abstract}
\maketitle
\tableofcontents
\section{Introduction}\label{introduction}
In this note we deal with a a Karoubian ${\mathbbm Q}$-linear tensor category $\fraka$ and the subcategory $\rep{\fraka}{G}$ of representations of a group $G$ in $\fraka$. That is the non-full subcategory of objects $X$ with an action $G\to {\rm Aut}(X)$. Morphisms in this category are $G-$equivariant morphisms.

Consider the case where $G$ is the symmetric group $\Sigma_n$. Let $R_n({\fraka})$ be the Grothendieck group of $\rep{\fraka}{\Sigma_n}$. As in the classical case there is the ring of $\fraka-$representation of symmetric groups
\[
R_\fraka=\coprod_{n\geq 0} R_n({\fraka}).
\]
This ring can be shown to be isomorphic to $R\otimes_{\mathbbm Z}K_0(\fraka)$ where $R$ is the usual ring of ${\mathbbm Q}-$representations of symmetric groups. That is the representation ring of symmetric groups in $\fraka$ is in one to one correspondence with the ring of symmetric functions with coefficients in $K_0(\fraka)$.

In most cases $\fraka$ has more structures than just direct sums and tensor products. For general additive $\fraka$ as above there is a natural $\lambda$-ring structure on $K_0(\fraka)$. Here we are concerned with well-definedness of this structure in the case of abelian categories and in some special cases of triangulated categories. In this case $K_0(\fraka)$ is taken with respect to exact sequences or triangles. We do not make any remark for the case of general triangulated categories. It should be mentioned that for most of the examples of triangulated categories we have in mind, the $\lambda$-ring structure is well-defined.\footnote{See: Biglari, S. \emph{Lambda ring structure on the Grothendieck ring of mixed motives: Preliminary version}, Preprint. 2010. Available from author's website.} 

In the last section we consider the relationship via general characteristic $X\mapsto \overline{\chi}_X$ between the Grothendieck ring of the category of representations in $\fraka$ of a group $G$ and the ring of $\Lambda\bigl({\rm End}_{\fraka}({\mathbbm 1})\bigr)$-valued central functions on $G$. This characteristic is also compatible with respective intrinsic $\lambda$-ring structures or equivalently Adams operations $\psi_n$. It is this structure that helps to compute the trace of arbitrary power of endomorphisms inductively. That is
\[
\overline{\chi}_{\psi_n(X)}(g)=\overline{\chi}_{X}(g^n).
\]
The case of some particular (triangulated) examples shall be considered in forthcoming notes.
\section{Ring and category of representations}\label{sec:rep-ring-i}
Let $\fraka$ be a ${\mathbbm Q}-$linear tensor category in the sense of~\cite[I.0.1.2, I.2.4]{saavedra}.  In particular such a category as $\fraka$ is equipped with an associative, commutative, and unital tensor structure. We assume that ${\fraka}$ is essentially small and Karoubian, i.e. every projector $p=p^2\in {\rm End}(X)$ has a kernel. A $G-$object or representation of a group $G$ in $\fraka$ is a pair $(X, \xi_X)$ consisting of an object $X$ of ${\mathfrak A}$ and a homomorphism $\xi_X\colon G\to \text{Aut}_{{\mathfrak A}}(X)$ written as $a\mapsto a_X$ or just $a\mapsto a$. A $G-$morphism (or $G-$equivariant morphism) between two such representations $(X, \xi_X)$ and $(Y, \xi_Y)$ is a morphism $f\colon X\to Y$ such that $a_Y\circ f=f\circ a_X$ for all $a\in G$.
\begin{definition}
The subcategory $\text{Rep}_{{\mathfrak A}}(G)$ of $\fraka$ is defined to have representations of $G$ in ${\mathfrak A}$ as objects and $G-$morphisms as morphisms.
\end{definition}
We immediately note that $\text{Rep}_{{\mathfrak A}}(G)$ is a ${\mathbbm Q}-$linear tensor subcategory where the action of $G$ on $Z=X\otimes Y$ is given by $a\mapsto a_X\otimes a_Y$. Note that the commutativity and associativity constraints are by functoriality $G-$equivariant. Moreover the unit ${\mathbbm 1}$ provides a trivial representation of $G$ via the homomorphism ${\mathbbm Q}G\to {\mathbbm Q}\to {\rm End}_\fraka({\mathbbm 1})$.
\begin{lemma}\label{thm:rep-abelian}
If $\fraka$ is a ${\mathbbm Q}-$linear abelian category. Then $\rep{\fraka}{G}$ has a unique structure of a ${\mathbbm Q}-$linear abelian category such that the forgetful functor $\rep{\fraka}{G}\to \fraka$ is exact and reflects exactness.
\end{lemma}
\begin{proof}
Let $f\colon X\to Y$ be a morphism in $\rep{\fraka}{G}$. Let $K$ be a kernel of this morphism in $\fraka$. For $a\in G$, let $a_K$ be the unique endomorphism of $K$ making the diagram
\[
\xymatrix{
K\ar@{-->}[d]^-{a_K} \ar[r] & X\ar[d]^-{a_X} \ar[r]^-{f} & Y\ar[d]^-{a_Y}\\
K\ar[r] & X\ar[r]^-{f} & Y
}
\]
commutative. Uniqueness of such morphisms show that the assignment $a\mapsto a_K$ defines a group homomorphism $G\to {\rm Aut}_\fraka(K)$. Similarly cokernel of any morphism in $\rep{\fraka}{G}$ is defined and $\rep{\fraka}{G}\to \fraka$ reflects exactness. 
\end{proof}

\begin{remark}
Many other structures of $\fraka$, when applicable, induce the same structure on $\rep{\fraka}{G}$. We implicitly use such induced structures. However if $\fraka$ is an arbitrary triangulated category, then the above argument does not give $\rep{\fraka}{G}$ a triangulation due to the fact that the cone construction $f\mapsto {\rm cone}(f)$ does not produced any good representable functor. However in this case $\rep{\fraka}{G}$ has one artificial but important structure that we shall consider later (see below~\S\ref{sec:rep-ring-ii}).
\end{remark}
\begin{example}\label{ex:classic}
We consider the following example. For $n\geq 0$ denote by $R_n$ the Grothendieck group of the semi-simple abelian category of finitely generated ${\mathbbm Q}\Sigma_n-$modules. This is a free abelian group on the set of isomorphism classes of irreducible objects (which is indexed by the set of partitions of $n$). Let $0$ have one and only one partition. The ring of representations of symmetric groups is defined to be
\begin{equation}\label{def:ring-rep}
R:=\coprod_{n\geq 0} R_n.
\end{equation}
This is a commutative, associative, and unital ring with the multiplication of classes of a ${\mathbbm Q}\Sigma_p-$module $V$ and a ${\mathbbm Q}\Sigma_q-$module $W$ given by 
\[
\cl{}(V)\cl{}(W)=\cl{}\bigl(\text{Ind}^{\Sigma_{p+q}}_{\Sigma_{p}\times \Sigma_{q}}\bigl(V\otimes W\bigr)\bigr).
\]
Note that each $R_n$ is also a ring but we ignore this structure. Now let $\Lambda_n$ be the ring of symmetric polynomials in $n$ variables, i.e. the subring of $\Z[x_1,\cdots,x_n]$ consisting of elements $f$ with $\sigma f=f$ for all $\sigma\in \Sigma_n$. The substitution $x_{n+1}=0$ defines a morphism $\Lambda_{n+1}\to \Lambda_n$ of graded rings. The graded ring of symmetric functions is defined to be the projective limit 
\[
\Lambda :=\underset{\longleftarrow}{\rm lim}^{\rm gr}\ \Lambda_n
\]
taken in the category of graded rings. The remarkable result is that there is a ring isomorphism $R\simeq \Lambda$. An isomorphism is given by $\class{}{V_\pi}\mapsto s^{\pi}$ where $s^\pi$ is the Schur function associated to $\pi$. Proofs of this can be found in~\cite[\S 1]{Atiyah1966} and ~\cite[III]{knutson}. As we shall see this is extended to more general settings.
\end{example}

Let $\fraka$ be as above an arbitrary Karoubian ${\mathbbm Q}-$linear tensor category. Let $V$ be a finite dimensional ${\mathbbm Q}-$representation of $\Sigma_n$, i.e. an object of $\rep{{\mathbbm Q}-{\rm fgmod}}{\Sigma_n}$. Let $Y$ be an object of $\fraka$. Define the object $V\otimes Y$ of $\fraka$ to be a direct sum of $\dime{V}$ copies of $Y$. This turns out to be a functorial assignment, that is $V\otimes X$ represents a functor from $\fraka$ to ${\mathbbm Q}-{\rm fgmod}$. Therefore there is a natural action of $\Sigma_n$ on $V\otimes Y$. This gives a functor
\[
T\colon \rep{{\mathbbm Q}-{\rm fgmod}}{\Sigma_n}\times \fraka\to \rep{\fraka}{\Sigma_n}.
\]
In fact from tenosr product of categories, but we do not need this. Now let $X$ be an object in $\rep{\fraka}{\Sigma_n}$. Let $V$ be any irreducible finite dimensional ${\mathbbm Q}-$representation of $\Sigma_n$. By basic results of representation theory of $\Sigma_n$ in characteristic zero we know that $V=\pi^V {\mathbbm Q}\Sigma_n$ for an idempotent $\pi^V={\pi^V}\circ {\pi^V}\in {\mathbbm Q}\Sigma_n$. Define
\begin{equation}\label{def:s-lambda}
S_V(X)={\rm Im}_\fraka (\pi^V_X\colon X\to X).
\end{equation}
A better notation is obviously $\underline{\rm Hom}_{\Sigma_n}(V, X)$ but we use the above notation. We let $Irr_n$ be the set isomorphism classes of irreducible finite dimensional ${{\mathbbm Q}}$-representation of $\Sigma_n$.
\begin{lemma}\label{thm:aw-cat}
There is a natural $\Sigma_n-$isomorphism
\begin{equation}\label{aw:cat}
X\simeq \coprod_{V\in Irr_n} V\otimes S_V(X).
\end{equation}
\end{lemma}
\begin{proof}
This is a formal consequence of the corresponding result in the case $\fraka$ is the category ${\mathbbm Q}\Sigma_n-{\rm fgmod}$. See also~\cite{deligne-2002} where $\Sigma_n-$objects of the form $X=Y^{\otimes n}$ are considered. 
\end{proof}
\begin{example}\label{ex:spi}
For $\fraka$ as above, let $T_n\colon \fraka\to \rep{\fraka}{\Sigma_n}$ be the functor $X\mapsto X^{\otimes n}$. For any irreducible finite dimensional representation $V_\pi$ of $\Sigma_n$ over ${\mathbbm Q}$ (where $\pi$ is a partition of $n$), and following the standard notations we denote 
\[
S_{V_\pi}(T_n(X))=:S_\pi(X).
\]
\end{example}

The isomorphism~\ref{thm:aw-cat} helps us to give alternative definitions for notions familiar from classical representation theory. For example let $p+q=n$, $W$ a representation of $\Sigma_p\times \Sigma_q$ over ${\mathbbm Q}$ and $Y$ an object of $\fraka$. Define
\begin{equation}\label{def:ind}
{\rm Ind}^{\Sigma_n}_{\Sigma_p\times \Sigma_q}(W\otimes Y):={\rm Ind}^{\Sigma_n}_{\Sigma_p\times \Sigma_q}(W)\otimes Y
\end{equation}
This is an object of $\rep{\fraka}{\Sigma_n}$. In particular if $X_p$ and $X_q$ are two respective representations of $\Sigma_p$ and $\Sigma_q$ in $\fraka$, then by~\ref{thm:aw-cat} the representation $X_p\otimes X_q$ of $\Sigma_p\times \Sigma_q$ is isomorphic to a direct sum of representations of the form $V\otimes Y$. Applying the above definition of induced modules to each summand, we obtain a representation of $\Sigma_n$ which is denoted by ${\rm Ind}^{\Sigma_n}_{\Sigma_p\times \Sigma_q}(X_p\otimes X_q)$.
\begin{remark}
The restriction to $G=\Sigma_n$ in the above discussions is only for conventional reasons and simplicity of formulas.
\end{remark}

\section{Operations on Grothendieck rings}\label{sec:ad-ho}
Let $\fraka$ be an essentially small Karoubian ${\mathbbm Q}-$linear tensor category. Let the notation below be as in~\ref{ex:spi}. It is easily seen that $K_0(\fraka)$ is an associative commutative unital ring. We use only the general results of~\cite[1]{deligne-2002}. This section overlaps with~\cite{Heinloth2007} and follows the suggestion from~\cite[1.]{fdmtm}.

For an element $X$ of the set of isomorphism classes of $\fraka$ define as in~\cite{fdmtm} the element
\begin{equation}\label{lambda:eq}
\lambda_{\Sigma}(X)=\sum_{\mu} \cl{}\bigl(S_\mu(X)\bigr)\otimes \cl{}(V_\mu)t^{|\mu|}\in  K_0(\fraka)_R[\![t]\!]^{\times}
\end{equation}
where the sum is taken over all partitions $\mu$ of $|\mu|\geq 0$, $V_\mu$ is the irreducible ${\mathbbm Q}\Sigma_{|\mu|}-$module corresponding to $\mu$, and $K_0(\fraka)_R:=K_0(\fraka)\otimes_\Z R$ for $R$ is the representation ring~\eqref{def:ring-rep}.

\begin{lemma}\label{glr-additive}
$X\mapsto \lambda_{\Sigma}(X)$ is a well-defined monomorphism of abelian groups from $K_0(\fraka)\to 1+tK_0(\fraka)_R[\![t]\!]$.
\end{lemma}
\begin{proof}
We must show that $\lambda_\Sigma ({X\oplus Y})=\lambda_\Sigma (X)\lambda_\Sigma (Y)$. We compare the coefficients of $t^n$. Let $\pi$ be a partition of $n$. By~\cite[Proposition 1.8]{deligne-2002} we have the formula
\begin{equation}\label{def:schur-dec}
S_\pi(X\oplus Y)\simeq \coprod_{\mu,\eta} S_{\mu}(X)\otimes S_{\eta}(Y)^{\oplus c^\pi_{\mu,\eta}}
\end{equation}
which is in turn a result of functorial isomorphism in~\ref{thm:aw-cat}. Note that the integers $c^\pi_{\mu,\eta}$ satisfy
\[
R_n\ni \induced{\Sigma_n}{\Sigma_p\times \Sigma_q}{V_\mu\otimes V_\eta}=\coprod_\pi V_\pi^{\oplus c^\pi_{\mu, \eta}}.
\]
Therefore $\sum_{|\pi|=n}\class{}{S_\pi(X\oplus Y)}\otimes \class{}{V_\pi}$ is exactly the coefficient of $t^n$ in $\lambda_\Sigma (X)\lambda_\Sigma (Y)$.
\end{proof}

With notations as in~\ref{glr-additive}, consider $\lambda_\Sigma$. Composing with natural projection onto $K_0(\fraka)\otimes \cl{}(V_\pi)\simeq K_0(\fraka)$, we obtain a well-defined (not a homomorphism) mapping $S_\pi\colon K_0(\fraka)\to K_0(\fraka)$ extending $S_\pi(\cl{}(X))=\cl{}\bigl(S_\pi(X)\bigr)$. We use the ring structure of $K_0(\fraka)$ and define the product $S_\mu S_\eta$ value-wise. A few of the following results can be derived differently, however we give direct proofs.
\begin{lemma}
For any $x,y\in K_0(\fraka)$ and each partition $\pi$ we have
\begin{equation}\label{sum-inverse}
S_\pi(-x)=(-1)^{|\pi|}S_{\pi^t}(x)\text{  and }S_\pi(x+y)=\sum_{\mu, \eta} c^\pi_{\mu,\eta}S_\mu(x) S_\eta(y).
\end{equation}
\end{lemma}
\begin{proof}
For each generator $\cl{}(V_\mu)$ of $R$ define ${\text{ev}_x}(V_\mu)$ to be $S_\mu(x)$. This gives a well-defined $\Z-$linear map
\begin{equation}\label{def:ev}
{\text{ev}_x}\colon R\to K_0(\fraka).
\end{equation}
First let $x=\cl{}(X)$. It follows from~\cite{deligne-2002}*{Proposition 1.6} that for partitions $\mu$ and $\eta$ we have
\[
S_\mu(x)S_\eta(x)=\sum_\pi c^\pi_{\mu,\eta}S_\pi(x).
\]
That is ${\text{ev}_x}$ is in fact a ring homomorphism for $x=\cl{}(X)$. By definition of the ring structure of $R$ and the identity $\lambda_\Sigma(x+y)=\lambda_\Sigma(x)\lambda_\Sigma(y)$ we obtain the second equality. For the first equality note that by making use of the second equality it is enough to show that for any object $X$ of $\fraka$ and any partition $\pi$ with $|\pi|>0$ we have in $K_0(\fraka)$
\[
\sum_{\mu,\eta}(-1)^{|\eta|}c^\pi_{\mu,\eta}\cl{}\bigl(S_\mu(X)\otimes S_{\eta^t}(X)\bigr)=0.
\]
To see this note that the left hand side is the image under ${\text{ev}_{\cl{}(X)}}$ of a similar sum in $\Lambda\simeq R$ of~\ref{ex:classic} with $S_\pi$'s replaced by corresponding Schur functions $s^\pi$. Now by the general results~\cite{knutson}*{Chap. I, III}, $\Lambda$ is the free special $\lambda-$ring on one variable $a_1$ and any function $s^\pi$ can be written as $S_\pi(a_1)$. Therefore the sum above is
\[
\sum_{\mu,\eta}(-1)^{|\eta|}c^\pi_{\mu,\eta}S_\mu(a_1)\otimes S_{\eta^t}(a_1)=S^\pi(a_1-a_1)=0
\]
\end{proof}
\begin{cor}\label{thm:ev-ring}
${\rm ev}_x$ is a ring homomorphism for any $x\in K_0(\fraka)$, that is
\[
S_\mu(x)S_\eta(x)=\sum_\pi c^\pi_{\mu,\eta}S_\pi(x).
\]
\end{cor}
Now define the homomorphism $\lambda\colon K_0(\fraka)\to K_0(\fraka)[\![t]\!]^{\times}$ by extending
\begin{equation}\label{def:lambda:eq:alt}
\lambda(\cl{}(X))=\sum \cl{}(\Alt{n}{X})t^n.
\end{equation}
Note that by~\ref{glr-additive} this is a well-defined homomorphism because product of summands in $K_0(\fraka)\otimes \class{}{V_{(n)^t}}$ lands in similar summand. We note that
\begin{proposition}\label{thm:special-additive}
Thus defined $\lambda$-ring structure~\eqref{def:lambda:eq:alt} on $K_0(\fraka)$ is special.
\end{proposition}
\begin{proof}
This is proved in~\cite[4.1]{Heinloth2007}.
\end{proof}

Now consider the category $K^b(\fraka)$ of bounded complexes of $\fraka$ with morphisms being those of complexes up to homotopy equivalence. This is a Karoubian ${\mathbbm Q}-$linear tensor triangulated category in all senses that could be understood. Moreover the embedding $\fraka\to K^b(\fraka)$ is a tensor functor.

Define $K_0\bigl(K^b(\fraka)\bigr)$ to be the quotient of free abelian group on the set of isomorphism classes of objects of $K^b(\fraka)$ by the subgroup generated by $\cl{}(Z)-\cl{}(X)-\cl{}(Y)$ for any distinguished triangle $X\to Z\to Y\to X[1]$. In particular $\cl{}(X)=\cl{}(X')$ for any homotopic objects $X\to X'$. Note that since $\fraka$ is Karoubian, the natural embedding $\fraka\to K^b(\fraka)$ by~\cite{Gillet1996}*{3.2.1} induces an isomorphism on Grothendiek groups.
\begin{remark}
Using the isomorphism $i_\ast\colon K_0(\fraka)\to K_0(K^b(\fraka))$ we may want to introduce a $\lambda-$ring structure on $K_0(K^b(\fraka))$. We show that this can be done on the level of triangulated category. 
\end{remark}

\begin{proposition}\label{wd:ka}
$\lambda_{\Sigma}$ is well-defined on $K_0\bigl(K^{b}(\fraka)\bigr)$.
\end{proposition}
\begin{proof}
We must show that if $\Delta:Y\to X\to Z\to Y[1]$ is a distinguished triangle in $K^{b}(\fraka)$, then
\[
\lambda_{\Sigma}(X)=\lambda_{\Sigma}(Y)\lambda_{\Sigma}(Z).
\]
It is clear that if $X\simeq X'$ in $K^{b}(\fraka)$, then $\lambda_\Sigma(X)=\lambda(X')$. By definition, we may assume that $\Delta$ is a standard distinguished triangle, i.e. $Z^n=Y^{n+1}\oplus X^n$ with standard differentials. It is then enough to show that for each bounded complex $X$ we have
\begin{equation}\label{trk-lambda-prod}
\lambda_{\Sigma}(X)=\prod_{n\in \Z}\lambda_{\Sigma}(X^n)^{(-1)^n}.
\end{equation}
To see this let $\gr{S}\colon C^{b}(\fraka)\to K^{b}(\fraka)$ be the functor of the associated graded object of the dummy filtration, i.e. 
\[
X\mapsto \coprod_{n\in \Z}X^n[n].
\]
Note that $\cl{}(X)=\cl{}(\gr{S}(X))$ and more generally for each partition $\pi$ we have $\cl{}(\gr{S}( S_\pi(X)))\simeq \cl{}(S_\pi(\gr{S}(X)))$. In particular $\lambda_\Sigma (X)=\lambda_\Sigma (\gr{S}(X))$. Now the result follows from~\cite{fdmtm}*{Proposition 3.2} and~\eqref{sum-inverse} that
\[
\lambda_\Sigma (X[1])=\lambda_\Sigma  (X)^{-1}.
\]
\end{proof}
As in the case of $K_0(\fraka)$ in~\eqref{def:lambda:eq:alt} define $\lambda$ on the Grothendieck ring of $K^{b}(\fraka)$ to $K_0\bigl(K^{b}(\fraka)\bigr)[\![t]\!]^{\times}$ by extending
\[
\lambda(\class{}{X})=\sum \cl{}(\Alt{n}{X})t^n.
\]
\begin{proposition}\label{trian-add-lambda}
Thus defined $\lambda$ gives $K_0\bigl(K^{b}(\fraka)\bigr)$ the structure of a special $\lambda-$ring and the isomorphism $K_0(\fraka)\to K_0\bigl(K^{b}(\fraka)\bigr)$ respects the $\lambda-$ring structures.
\end{proposition}
\begin{proof}
This is straightforward. As it is the case for $\lambda-$rings, any equation that we need, holds once it is checked for objects of the form $\class{}{X}$. Let $i_\ast$ be the isomorphism of $K_0$ groups induced by $\fraka\to K^b(\fraka)$. This induces an isomorphism on respective rings of power series, again denote by $i_\ast$. Note that $\lambda\circ i_\ast=i_\ast \circ \lambda$ by~\eqref{trk-lambda-prod}. That is $i_\ast$ respects the $\lambda$-ring structures. This and the additive case~\ref{thm:special-additive} shows that the structure is special.
\end{proof}
\begin{example}
Let $X$ be an object of the Karoubian ${\mathbbm Q}-$linear tensor category $\fraka$. It follows from~\ref{trian-add-lambda} that via the isomorphism $K_0(\fraka)\to K_0\bigl(K^b(\fraka)\bigr)$ we have
\[
\lambda (X)^{-1}=\lambda_{K^b(\fraka)}(X[1])=\sum \class{}{\Sym{n}{X}} (-t)^n.
\]
 
\end{example}

\section{Abelian categories and their derived categories}\label{sec:ab-der}
Now we consider the case of abelian categories and their derived categories. We let $\fraka$ be an essentially small abelian ${\mathbbm Q}-$linear tensor category, we assume that $\otimes$ is biexact. Rigid tensor abelian category in the sense of~\cite{saavedra}*{I, \S 5} is in our practices the most important case. 

Let $K_0(\fraka)$ be the Grothendieck group of $\fraka$; defined as a quotient of $K_0(\fraka_{\rm add})$, where $\fraka_{\rm add}$ is the underlying additive subcategory of $\fraka$, by the subgroup generated by $\class{}{Z}-\class{}{X}-\class{}{Y}$ for any exact sequence $0\to X\to Z\to Y\to 0$. Consider the assignment $X\mapsto \lambda_\Sigma (X)$ with the same definition as in the previous case~\eqref{lambda:eq}.
\begin{lemma}\label{thm:abelian-lambda}
$\lambda_\Sigma$ is well-defined on $K_0(\fraka)$, that is \[{\rm cl}\bigl(S_\pi(Y)\bigr)={\rm cl}\bigl(S_\pi(X\oplus Z)\bigr)\] in $K_0(\fraka)$ for any exact sequence $0\to X\to Y\to Z\to 0$ and any partition $\pi$.
\end{lemma}
\begin{proof}
This follows from the method used in the proof of~\cite{deligne-2002}*{1.19}. More precisely $S_\pi(Y)$ has a finite increasing filtration whose object of graded pieces is $S_\pi(Y\oplus Z)$.
\end{proof}
\begin{remark}\label{rem:abelian-derived}
Let us denote by $i\colon \fraka\to D^b(\fraka)$ the natural tensor functor. 
The following diagram
\begin{equation*}\label{eq:abelian-derived}
{\xymatrix{
{\rm Obj}(\fraka)\ar[r]^-{i}\ar[d]^-{\lambda_\Sigma} & {\rm Obj}\bigl(D^b(\fraka)\bigr)\ar[d]^-{\lambda_\Sigma}\\
K_0(\fraka)[\![t]\!]\ar[r]^-{i_{\ast}} & K_0\bigl(D^b(\fraka)\bigr)[\![t]\!]
}}\raisetag{-100pt}
\end{equation*}
where $i_\ast:=K_0(i)$ is commutative.
\end{remark}

\begin{proposition}\label{thm:der-lambda}
Let $\fraka$ be an abelian ${\mathbbm Q}-$linear tensor category. Then $\lambda_{\Sigma}$ is well-defined on $K_0\bigl(D^{b}(\fraka)\bigr)$.
\end{proposition}
\begin{proof}
The proof is similar to~\ref{wd:ka}. In this case consider $\gr{\tau}$ defined by
\[
X\mapsto \coprod_{n\in \Z} H^n(X)[-n].
\]
Note that for each complex $Z$ we have $\class{}{Z}=\class{}{\gr{\tau}Z}$ in $K_0\bigl(D^{b}(\fraka)\bigr)$. Moreover $\gr{\tau}(Z^{\otimes n})\simeq \gr{\tau}(Z)^{\otimes n}$ by standard K\"unneth formula. The simple but important observation is that this isomorphism is defined in $D^b(\fraka)$ and is in fact $\Sigma_n-$equivariant. See for example~\cite{kfttc} for a proof. By~\ref{rem:abelian-derived} we have 
\begin{equation}\label{lambda-h}
 \begin{split}
 \lambda_\Sigma (Z) & = \lambda_\Sigma ({\gr{\tau} Z})\\
		   & = \prod \lambda_\Sigma \bigl(H^n(Z)[-n]\bigr)\\
		   & = \prod \lambda_\Sigma \bigl(H^n(Z)\bigr)^{(-1)^n}.
 \end{split}
\end{equation}
Let $X\to Z\to Y\to X[1]$ be a distinguished triangle. It is clear that if $Z\simeq X\oplus Y$, then $\lambda_\Sigma (Z)=\lambda_\Sigma (X)\lambda_\Sigma (Y)$. For the general case consider the long exact sequence
\[
\cdots\to H^n(X)\to H^n(Z)\to H^n(Y)\to H^{n+1}(X)\to\cdots
\]
and divide it into short exact sequences. Use the above formula~\eqref{lambda-h} and~\ref{thm:abelian-lambda} to conclude that $\lambda_\Sigma (Z)=\lambda_\Sigma (X\oplus Y)=\lambda_\Sigma (X)\lambda_\Sigma (Y)$
\end{proof}
\begin{remark}
We have not used the standard presentation of a distinguished triangle in $D^b(\fraka)$ to prove~\ref{thm:der-lambda}. The proof above can be utilized in a slightly more general setting of tensor triangulated categories with compatible $t-$structure considered in~\cite{kfttc}.
\end{remark}

As in the previous cases consider the homomorphism $\lambda$ on $K_0\bigl(D^{b}(\fraka)\bigr)$ defined by linear extension of 
\begin{equation}\label{lambda:eq:alt}
\lambda(\cl{}(X))=\sum \cl{}(\Alt{n}{X})t^n.
\end{equation}
\begin{proposition}\label{der-lambda-special}
$\lambda$ defines a $\lambda-$ring structure on $K_0\bigl(D^{b}(\fraka)\bigr)$ and the diagram
\begin{equation}
{\xymatrix{
K_0(\fraka)\ar[r]^-{i_{\ast}}\ar[d]^-{\lambda} & K_0\bigl(D^b(\fraka)\bigr)\ar[d]^-{\lambda}\\
K_0(\fraka)[\![t]\!]\ar[r]^-{i_{\ast}} & K_0\bigl(D^b(\fraka)\bigr)[\![t]\!]
}}
\end{equation}
is a commutative diagram of special $\lambda-$rings.
\end{proposition}
\begin{proof}
Set $R:=K_0(\fraka)$ and $S:=K_0\bigl(D^{b}(\fraka)\bigr)$. From~\ref{thm:der-lambda} it follows that $\lambda_S\colon S\to S[\![t]\!]$ is well-defined and hence $\lambda$ gives a $\lambda-$ring structure to $S$. Note that by the proof of~\ref{thm:der-lambda} we can write $\lambda (X)$ as the power series product (i.e. sum in $\Lambda(S)$) of $\lambda(H^nX)^{(-1)^n}$. This immediately shows that $\lambda_S$ commutes with $\lambda$-structures on $S$ and $\Lambda (S)$. That is the diagram above is commutative. Finally note that
\[
\lambda_S (X\otimes Y)=\lambda_S(X)\ast\lambda_S (Y)
\]
where $\ast$ is the ring multiplication of $\Lambda(S):=1+tS[\![t]\!]$ and $\lambda_S\colon S\to \Lambda (S)$ is the $\lambda-$structure. Commutativity of the diagram proves the identity above.  
\end{proof}
\begin{remark}
Let $D$ be a ${\mathbbm Q}-$linear tensor triangulated category. We do not know whether $K_0(D)$ is a $\lambda-$ring with the usual definition of $\lambda$ as in~\eqref{lambda:eq:alt}. That is if the assignment~\eqref{lambda:eq} (or equivalently~\eqref{def:lambda:eq:alt}) is well-defined on $K_0(D)$. What the above shows is that this is the case for homotopy category of complexes, derived category of an abelian tensor category, and any category $D$ with compatible $t-$structure. More generally, a localization of tensor triangulated categories (when preserving tensor structure) should be considered. 
\end{remark}

\section{Representation rings of symmetric groups}\label{sec:rep-ring-ii}
In this section we consider the case of triangulated categories. Needless to say that additive and abelian categories can be treated likewise. We let $\fraka=D$ be an essentially small ${\mathbbm Q}$-linear tensor triangulated category. Consider $\rep{D}{\Sigma_n}$. This is by the above discussion a ${\mathbbm Q}$-linear tensor category. We define $K^D_0\bigl(\rep{D}{\Sigma_n}\bigr)$ to be the quotient of the free abelian group on the set of isomorphism classes of objects of $\rep{D}{G}$ by a subgroup generated by $\cl{}(Z)-\cl{}(X)-\cl{}(Y)$ for any distinguished triangle $X\to Z\to Y\to X[1]$ all whose vertices and morphisms are in $\rep{D}{\Sigma_n}$. We call such a triangles to be an exact sequence in $\rep{\fraka}{\Sigma_n}$. As in~\S\ref{sec:rep-ring-i} we let $R_n$ be the Grothendieck group of the abelian category $\rep{{\mathbbm Q}-{\rm fgmod}}{\Sigma_n}$. Define
\begin{equation}\label{def:h}
h\colon K_0(D)\otimes R_n\to K_0^{D}(\rep{D}{\Sigma_n})
\end{equation}
by $\cl{}(X)\otimes \cl{}(V)\mapsto \cl{}(V\otimes X)$.
\begin{proposition}\label{thm:h}
$h$ is an isomorphism of rings.
\end{proposition}
\begin{proof}
First we show that $h$ is well-defined. It is enough to show that the functor
\[
T\colon \rep{{\mathbbm Q}-{\rm fgmod}}{\Sigma_n}\times D \to \rep{D}{\Sigma_n},\quad (V,X)\mapsto V\otimes X
\]
respects exactness in both arguments. Respecting the exactness in the first arguments amounts to respecting finite direct sum and hence trivial. For the second argument, let $\Delta:X\to Z\to Y\to X[1]$ be an exact triangle and $V$ a finite dimensional ${\mathbbm Q}-$representation of $\Sigma_n$. The sequence
\[
V\otimes X\to V\otimes Z\to V\otimes Y\to V\otimes X[1]
\]
is a direct sum of $\dime{V}$ copies of $\Delta$ and hence exact. Therefore $h$ is well-defined. Next we define an inverse of $h$. Let $Irr_n$ be a finite set of isomorphic classes of irreducible ${\mathbbm Q}\Sigma_n$ modules $V$ such that $\cl{}(V)$'s form a basis of free ${\mathbbm Z}-$module $R_n$. Define
\[
g\colon K^D_0(\rep{D}{\Sigma_n})\to K_0(D)\otimes R_{\mathbbm Q}(\Sigma_n)
\]
by $\cl{}(X)\mapsto \sum \cl{}(S_{V} (X))\otimes\cl{}(V)$ where $S_V(X)$ is defined in~\eqref{def:s-lambda}. Let us show that $g$ is well-defined. Let $\Delta:X\to Z\to Y\to X[1]$ be an exact sequence in $\rep{D}{\Sigma_n}$. By definition $\Delta$ is distinguished in $D$. Note that with notations as in~\eqref{def:s-lambda} the functor $U\mapsto S_{V}(U)$ is nothing but the functor $\rm{Im}_D(\pi^V_U)$ for an idempotent $\pi^V\in {\mathbbm Q}\Sigma_n$ with $V=\pi^V{\mathbbm Q}\Sigma_n$. Therefore the sequence $S_{V} (\Delta)$ being a direct summand of $\Delta$ is exact in $D$. Therefore $\cl{}(S_{V} (Z))=\cl{}(S_{V} (X))+\cl{}(S_{V} (Y))$. This means that $g$ is well-defined. To prove the proposition it is enough to prove that $h\circ g$ and $g\circ h$ are respective identity maps. To see this note that for $V$ and $W$ in $Irr_n$ we have
\[
S_{V} (W\otimes X)=\begin{cases}X & \text{if }V=W\\0 &\text{otherwise}\end{cases} 
\]
therefore $g\circ h=\id{}$. Similarly note that $h\circ g=\id{}$ by~\ref{thm:aw-cat}.
\end{proof}
\begin{remark}\label{thm:h-abelian}
For an additive (resp. abelian) category $\fraka$, the group $K_0(\rep{\fraka}{G})$ is defined. The result above is valid in this case with a minor improvement in notations. 
\end{remark}

We define the ring of representation in $D$ of symmetric groups to be 
\[
R_{D}:=\coprod_{n\geq 0} K^D_0\bigl(\rep{D}{\Sigma_n}\bigr).
\]
Let $X_p$ (resp. $X_q$) be an object of $\rep{D}{\Sigma_p}$ (resp. $\rep{D}{\Sigma_q}$). Define the multiplication of their classes in $R_{D}$ by 
\[
\class{}{X_p}\class{}{X_q}:={\rm cl} \bigl({\rm Ind}^{\Sigma_{p+q}}_{\Sigma_p\times \Sigma_q}(X_p\otimes X_q)\bigr).
\]
\begin{lemma}
This defines a commutative associative unital ring structure on $R_D$.
\end{lemma}
\begin{proof}
Fix integer $p$ and $q$, let $n=p+q$ and consider the assignment
\begin{align*}
{\rm Ind}^{\Sigma_n}_{\Sigma_p\times \Sigma_q}\colon \rep{D}{\Sigma_p}\times \rep{D}{\Sigma_q}& \to \rep{D}{\Sigma_n}
\end{align*}
This is biexact. Therefore the multiplication on $K_0$'s is well-defined. We show that it is commutative and associative. Let $X$ (resp. $Y$) be a $\Sigma_p$ (resp. $\Sigma_q$) module object. It is enough to show that commutativity constraint $\tau\colon X\otimes Y\to Y\otimes X$ gives an $\Sigma_{p+q}-$isomorphism
\[
{\rm Ind}^{\Sigma_{p+q}}_{\Sigma_p\times \Sigma_q}(X\otimes Y)\simeq {\rm Ind}^{\Sigma_{p+q}}_{\Sigma_q\times \Sigma_p}(Y\otimes X).
\]
The proof of this is identical to the classical case~\cite[III, p. 128]{knutson}. Associativity is proved similarly.
\end{proof}
Next we want to define a $\lambda-$ring structure on $R_D$. For this note that the direct sum of all isomorphisms in~\ref{thm:h} corresponding to different $n$'s gives an isomorphism
\[
H\colon R\otimes K_0(D) \to R_D, \quad \class{}{V}\otimes \class{}{X}\mapsto \class{}{V\otimes X}.
\]
We immediately note that by definitions~\eqref{def:ind} after~\ref{thm:aw-cat} $H$ is an isomorphism of rings. For the rest of this section assume that $\lambda$ in~\eqref{def:lambda:eq:alt} is well defined on $K_0(D)$. For example $D$ can be taken as one of those in~\ref{trian-add-lambda} or~\ref{der-lambda-special}. Therefore we obtain a $\lambda-$ring structure on $R_D$. Note that the $\lambda-$ring structure on $R\otimes K_0(D)$ is the natural one; for $r\in R$ and $x\in K_0(D)$
\[
\lambda (r\otimes x)=\lambda (r)\otimes 1\ast 1\otimes \lambda (x).
\]
Here we have denoted for example the embedding $\Lambda (R)\to \Lambda (R\otimes K_0(D))$ by $f\mapsto f\otimes 1$. Collecting all these together we have

\begin{theorem}\label{thm:fundamental}
There is a natural isomorphism of special $\lambda-$rings
\[
R_D\simeq R\otimes_{\Z}K_0(D).
\]
\end{theorem}
\begin{example}
Let $D=D^b(\fraka)$ be the bounded derived category of a tensor abelian category over a field $k$ of characteristic zero. There is a natural isomorphism 
\[
\xi\colon K_0^D\bigl(\rep{D}{\Sigma_n}\bigr)\to K_0\bigl(\rep{\fraka}{\Sigma_n}\bigr)
\]
defined by $\class{}{X}\mapsto \sum (-1)^n\class{}{H^n(X)}$. Therefore for $\fraka$ a result similar to the above holds.
\end{example}
\begin{remark}
If $\fraka$ is merely additive, then in view of~\ref{thm:h-abelian} a result similar to~\ref{thm:fundamental} holds.
\end{remark}

The above implies that $\lambda_\Sigma$ in~\eqref{lambda:eq} gives a homomorphism $\mu$ defined by the composite
\[
K_0(\fraka)\xrightarrow{\lambda_\Sigma}{} 1+tK_0(\fraka)\otimes R_[\![t]\!]\xrightarrow{\Lambda(H)}{} \Lambda(R_\fraka)
\]
By the above results and~\ref{glr-additive} we conclude that $\mu$ is an injective homomorphism of abelian groups. From the definitions we have
\begin{lemma} $\mu(\class{}{X})=\sum_n \class{}{X^{\otimes n}}t^n$.
\end{lemma}
Note that in general situations $\mu(\class{}{X})(1-\class{}{X}t)\neq 1$.

\section{Characteristic series of general group representations}
For any ring $A$ let $FC(G, A)$ be the ring of central functions from $G$ to $A$, that is functions $t\colon G\to A$ with $t(gh)=t(hg)$ and where ring structure on $FC(G, A)$ is defined value-wise. If $A$ is a (special) $\lambda$ ring, then $FC(G, A)$ is also a (special) $\lambda$ ring with value-wise definitions $\lambda^i(f)(g)=\lambda^i(f(g))$.

Let $\fraka$ be a ${\mathbbm Q}-$linear tensor category as in~\S\ref{sec:rep-ring-i} in which every object has a dual. In such a (rigid) category there is defined a ${\mathbbm Q}$-linear trace
\[
{\rm tr}\colon {\rm End}(X)\to {\rm End}({\mathbbm 1})\quad f\mapsto {\rm tr}(f;X).
\]
See~\cite[1.16]{deligne-2002}. Now let $G$ be a group (or monoid) and consider $\rep{\fraka}{G}$. As in the classical case define $\overline{\chi}\colon K_0\bigl(\rep{\fraka}{G}\bigr)\to FC(G,\Lambda(k))$ by
\begin{equation}\label{def:characteristic}
\class{}{X}\mapsto \overline{\chi}_X:=\bigl(g\mapsto \sum_n \trace{g;\Alt{n}{X}}t^n\bigr) 
\end{equation}
\begin{proposition}\label{thm:characteristic}
$\overline{\chi}$ is a $\lambda-$homomorphism of special $\lambda-$rings.
\end{proposition}
\begin{proof}
This is almost trivial from discussions above. Namely, first note that for general $\fraka':=\rep{\fraka}{G}$ the assignment
\[
\class{}{X}\mapsto \bigl(g\mapsto \trace{g;X}\bigr)
\]
is a ring homomorphism denoted ${\rm tr}_G\colon K_0(\fraka')\to FC(G,k)$. For example note that the action of $G$ on $X\otimes Y$ is nothing but $g\mapsto g_X\otimes g_Y$ and hence $\trace{g;X\otimes Y}$ equals  $\trace{g;X}\trace{g;Y}$. By the previous results the homomorphism
\[
\lambda\colon K_0(\fraka')\to 1+tK_0(\fraka')[\![t]\!]=\Lambda (K_0(\fraka'))
\]
is a $\lambda-$homomorphism of $\lambda-$rings. By functoriality of $A\mapsto \Lambda (A)$ we conclude that $\Lambda({\rm tr}_G)\circ \lambda$ is a $\lambda-$homomorphism of $\lambda-$rings. But this composition is nothing but $\overline{\chi}$. 
\end{proof}

For any object $X$ define $d_X:=\trace{\id{};X}\in {\rm End}_{\fraka}({\mathbbm 1})$. For any partition $\pi$ of $n=|\pi|$ we let $d_\pi$ be the polynomial whose value at an integer $m$ is exactly the dimension of $S_\pi(W)$ where $W$ is a ${\mathbbm Q}-$vector space of dimension $m$. The non-trivial fact is that due to Weyl character formula such a polynomial exists.
\begin{lemma}\label{thm:trace-id-schur} For any object $X$ and any partition $\pi$ we have
\[
\trace{\id{}; S_\pi(X)}=d_{\pi}(d_X).
\]
\end{lemma}
\begin{proof} (Based on~\cite[7.1.2]{Deligne1990}). For $n=|\pi|$, we have 
\[
\trace{\id{}; S_\pi(X)}=\frac{1}{n!}\sum_{\sigma\in \Sigma_n} a^\pi_\sigma \trace{\sigma; X^{\otimes n}}
\]
for certain $a^\pi_\sigma\in {\mathbbm Z}$. By~\cite{Deligne1990} we know that $\trace{\sigma; X^{\otimes n}}$ is a power of $\trace{\id{}; X}$ by the number of disjoint-cyclic decomposition of $\sigma$. Therefore there is a polynomial in ${\mathbbm Q}[U]$ independent from $\fraka$ (and $G$ and $X$) such that $\trace{\id{}; S_\pi(X)}=P(\trace{\id{};X})$. Using this for $\fraka$ the category of finitely generated ${\mathbbm Q}-$modules and $G=1$ we see that $P(m)=d_\pi(m)$ for all positive integers $m$. We obtain the result.
\end{proof}
\begin{example}
In particular the above result shows that in general
\[
\overline{\chi}_X(1)=1+d_Xt+\cdots+\binom{d_X}{n}t^n+\cdots.
\]
\end{example}
\begin{example}\label{ex:character-trace}
Let $V$ be a finite dimensional (virtual) ${\mathbbm Q}-$representation of $\Sigma_n$, i.e. an element of $R_n$. Applying the morphism ${\rm ev}_X\colon R\to K_0(\fraka')$ from~\eqref{def:ev}, we obtain an element which we denote by $S_V(X^{\otimes n})$ or $S_V(X)$. Let $f\in G$. If $\chi_V$ is the (virtual) character of $V$ then
\begin{equation}
\trace{f;S_V(X)}=\frac{1}{n!}\sum_{\sigma\in \Sigma_n}\chi_V(\sigma^{-1})\trace{\sigma f^{\otimes n}; X^{\otimes n}}.
\end{equation}
To see this note that the formula is ${\mathbbm Z}$-linear on $V$. It is enough to prove this for irreducible $V=V_\pi$. We know that $V_\pi\otimes S_{V_\pi}(X)$ is the image of the projector
\[
q:=\frac{\chi_\pi(1)}{n!}\sum \chi_\pi (\sigma^{-1})\sigma\colon X^{\otimes n}\to X^{\otimes n}
\]
and since $\chi_\pi(1)=\dime{V_\pi}\neq 0$ we have
\begin{equation*}
\begin{split}
\trace{S_{V_\pi}(f);S_{V_\pi}(X)} & = \frac{1}{\chi_\pi(1)}\trace{\id{}\otimes S_{V_\pi}(f);V_\pi\otimes S_{V_\pi}(X)}\\
& = \frac{1}{\chi_\pi(1)}\trace{qf^{\otimes n};X^{\otimes n}}\\
& = \frac{1}{n!}\sum_{\sigma\in \Sigma_n}\chi_\pi(\sigma^{-1})\trace{\sigma f^{\otimes n}; X^{\otimes n}}.
\end{split}
\end{equation*}

\end{example}

Now we consider the Adams operations. For an integer $n\geq 1$ and a $\lambda-$ring $A$, define the $n-$th Adams operation $\psi_n\colon A\to A$ to be given by
\begin{equation}\label{def:adams}
-t\frac{d\lambda (x)}{dt}/\lambda(x)=\sum_{n\geq 1} \psi_n(x)(-t)^n.
\end{equation}

We shall obtain a formula that is similar to the one in~\cite[9.1, Ex. 3]{Serre71} for the classical case of $\fraka=k-{\rm fgmod}$\footnote{I thank B. Kahn for pointing out to me this reference.}. Note that the following result should be related to what the remark~\cite[15.6]{Kahn2009a} suggests.
\begin{theorem}\label{thm:adams-g} 
For any object $X$ and any $g\in G$, we have
\[
\overline{\chi}_{X}(g^n)=\psi_n\bigl(\overline{\chi}_{X}(g)\bigr)
\]
\end{theorem}
\begin{proof}
This is a formal consequence of classical results. For any integer $m$ Let $L_m$ be an element of ${\mathbbm Q}\Sigma_n$ which under the isomorphism $\Lambda\simeq R$ corresponds to the Newton power sum function $p_m=x_1^m+x_2^m+\cdots$. It is known (\cite[III, \S 2]{knutson}) that (the character of) $L_m$ is given by $\sigma\mapsto m$ if $\sigma$ is a $m-$cycle and zero otherwise. We use~\ref{ex:character-trace} and consider the corresponding element 
\[
q=\frac{1}{n!}\sum_{\sigma} L_{(n)}(\sigma^{-1})\sigma=\frac{1}{(n-1)!}\sum_{\text{$n-$cycles}}\sigma
\]
of ${\mathbbm Q}\Sigma_n$. Therefore applying the homomorphism ${\rm ev}_X\colon R\to K_0(\fraka')$ to $L_m$ with $\fraka':=\rep{\fraka}{G}$ we have similar to~\ref{thm:trace-id-schur} (in the proof of~\cite[7.1.2]{Deligne1990} put $u_i=g$ for all $i$) the equalities
\begin{equation}
 \begin{split}
\trace{g;{\rm ev}_X(L_{(m)})} & = \trace{qg^{\otimes m};X^{\otimes m}}\\
		   & = \frac{1}{(m-1)!}\sum_{\text{$m-$cycles}} \trace{g^m;X}\\
		   & = \trace{g^m;X}.
 \end{split}
\end{equation}
Now it is enough to apply ${\rm tr}_G\circ {\rm ev}_X$ to the identity
\[
\sum_{k=1}^n (-1)^k p_ke_{n-k}=-ne_n
\]
in $\Lambda\simeq R$ where $e_i$'s correspond to alternating representations and $p_j$'s are the above power sums. The result is an identity between coefficients in~\eqref{def:adams} so that
\[
-t\frac{d\overline{\chi}_{X}(g)}{dt}/{\overline{\chi}_{X}(g)}=\sum_{n\geq 1} \trace{g^n;X}(-t)^n
\]
in the $\lambda-$ring $\Lambda(k)=1+tk[\![t]\!]$. The rest is a formal consequence of definitions.
\end{proof}
\begin{remark}
Note that the proof does not really use $G$ as whole. The formula is valid for any endomorphism. Also note that the formula in~\ref{thm:adams-g} states nothing but
\[
\overline{\chi}_{\psi_n(X)}(g)=\overline{\chi}_{X}(g^n).
\]

\end{remark}


\end{document}